\documentclass[12pt]{article}
\usepackage[english]{babel}
\usepackage{graphicx}
\usepackage{amsmath,amsthm,amssymb}
\usepackage{times}
\usepackage{enumerate}
\usepackage{amsmath}
\usepackage{amsfonts}
\usepackage{pdfsync}
\usepackage{cleveref}
\usepackage{autonum}
\usepackage[usenames, dvipsnames]{color}

\newcommand{\R}{{\mathbb R}}

\renewcommand{\L}{\Lambda}

\DeclareMathOperator{\Dom}{Dom}

\DeclareMathOperator{\AC}{W^{1,1}}

\DeclareMathOperator{\dist}{dist}

\def\titlerunning#1{\gdef\titrun{#1}}
\makeatletter
\def\author#1{\gdef\autrun{\def\and{\unskip, }#1}\gdef\@author{#1}}
\def\address#1{{\def\and{\\}\renewcommand{\thefootnote}{}%
\footnote {#1}}%
\markboth{\autrun}{\titrun}}
\makeatother
\def\email#1{e-mail: #1}
\def\subjclass#1{{\renewcommand{\thefootnote}{}%
\footnote{\emph{Mathematics Subject Classification (2010):} #1}}}
\def\keywords#1{\par\medskip
\noindent\textbf{Keywords.} #1}

\theoremstyle{definition}
\newtheorem{theorem}{Theorem}[section]
\newtheorem{corollary}[theorem]{Corollary}
\newtheorem{lemma}[theorem]{Lemma}

\newtheorem{proposition}[theorem]{Proposition}
\newtheorem*{theorem*}{Theorem}
\usepackage{amsthm}

\theoremstyle{definition}
\newtheorem{definition}[theorem]{Definition}
\newtheorem{remark}[theorem]{Remark}
\newtheorem{example}[theorem]{Example}

\newtheorem*{hypothesisH1'}{Hypothesis $\mathbf{(H_2')}$}
\newtheorem*{hypothesisH1''}{Hypothesis ($\mathcal S^{\infty}_{x_*}$)}
\newtheorem*{hypothesisS'}{Hypothesis ($\mathcal S^{\infty}_{x_*}$)}
\newtheorem*{hypothesisH'}{Hypothesis ($\mathcal A^{\infty}_{x_*}$)}

\newtheorem*{problemP'}{Problem \textbf{(P$'$)}}
\newtheorem*{GrowthGloc}{Growth Condition (G)}
\newtheorem*{ConditionH}{Growth Condition (H$_B^{\delta}(\chi)$)}

\newtheorem*{conditionS}{Condition (S)}
\newtheorem*{Notation}{Notation}
\newtheorem*{basicass}{Basic Assumptions and Notation}
\newtheorem*{AssumptionG}{Superlinearity}

\numberwithin{equation}{section}

\begin{document}
\titlerunning{Uniform boundedness of minimizers}
\title{Uniform boundedness for the optimal controls of a discontinuous, non--convex Bolza problem}
\author{Piernicola Bettiol
\and
Carlo Mariconda
}
\maketitle
\date
\address{Univ Brest,  UMR CNRS 6205, Laboratoire de Math\'ematiques de Bretagne Atlantique,
 6 Avenue Victor Le Gorgeu, Brest,
29200-F.; \email{piernicola.bettiol@univ-brest.fr }
\and
Carlo Mariconda (corresponding author), ORCID $0000-0002-8215-9394$),  Universit\`a
degli Studi di Padova,  Dipartimento di Matematica  ``Tullio Levi-Civita'',   Via Trieste 63, 35121 Padova, Italy;  \email{carlo.mariconda@unipd.it}}
\begin{abstract}
We consider a Bolza type optimal control problem of the form
\begin{equation}\min J_{t}(y,u):=\int_t^T\L(s,y(s), u(s))\,ds+g(y(T))\tag{P$_{t,x}$}\end{equation}
\text{\textbf{Subject to:} }
\begin{equation}\label{tag:admissible}\tag{D}\begin{cases}
y\in \AC([t,T];\R^n)\\y'=b(y)u\text{ a.e. } s\in [t,T], \,y(t)=x\\u(s)\in \mathcal U\text{ a.e. } s\in [t,T],\, y(s)\in \mathcal S\,\,\forall s\in [t,T],
\end{cases}
\end{equation}
where $\L(s,y,u)$ is locally Lipschitz in  $s$,  just Borel in $(y,u)$, $b$ has at most a linear growth and both the Lagrangian $\L$ and the running cost function $g$ may take the value $+\infty$. If $b\equiv 1$ and $g\equiv 0$ problem (P$_{t,x}$) is the classical one of the calculus of variations.  We suppose the validity a slow growth condition in $u$, introduced by Clarke in 1993, including Lagrangians of the type $\L(s,y,u)=\sqrt{1+|u|^2}$ and $\L(s,y,u)=|u|-\sqrt{|u|}$ and the superlinear case.
If $\L$ is real valued, any family of optimal pairs $(y_*, u_*)$ for (P$_{t,x}$)  whose energy $J_t(y_*, u_*)$ is equi-bounded as $(t,x)$ vary in a compact set, has $L^{\infty}$ -- equibounded optimal controls.
If $\L$ is extended valued, the same conclusion holds under an additional lower semicontinuity assumption on $(s,u)\mapsto\L(s,y,u)$ and on the structure of the effective domain. No convexity, nor local Lipschitz continuity is assumed on the variables $(y,u)$. As an application we obtain the local Lipschitz continuity of the value function under slow growth assumptions.
\end{abstract}
\section{Introduction}
A major issue arising in the basic problem of the calculus of variations is the Lipschitz regularity of the minimizers.
Providing positive answers on this issue is often a first step towards higher regularity properties, and it allows numerical methods to catch the value of the infimum.

We consider here optimal control problems, such as (P$_{t,x}$)-(D) below, imposing very weak assumptions on the Lagrangian $\L(s,y,u)$, where $s\in [t_0,T]$ (the time variable),  $y\in \R^n$ (the state variable) and $u\in \R^m$  (the control variable),  motivated by the fact that, starting from the calculus of variations case (i.e. when $b\equiv 1$, $u\in \R^n$) there are discontinuous and non-convex problems that admit existence of minimizers, even if the classical Tonelli's existence conditions are not satisfied.

In the calculus of variations setting several results appeared on the subject following Tonelli himself \cite{Tonelli}: we just mention  Clarke -- Vinter \cite{CVTrans}, Ambrosio -- Ascenzi -- Buttazzo \cite{AAB}, Cellina \cite{Cellina}.
In the autonomous case,  just superlinearity and even slower growths suffice to obtain Lipschitzianity of the minimizers, whether they exist among the absolutely continuous functions (Dal Maso -- Frankowska \cite{DMF}, Mariconda -- Treu \cite{MTLip}).

In the nonautonomous case growth conditions in general do not guarantee the Lipschitzianity of the minimizers. A celebrated example by Ball -- Mizel \cite{BM} shows that there are polynomial Lagrangians that satisfy  Tonelli's existence assumptions (convexity in the velocity variable and superlinearity) for which  even the Lavrentiev phenomenon occurs (i.e., the infimum of the functional among Lipschitz functions is strictly greater than the infimum taken over the absolutely continuous ones).
So, extra hypotheses are needed in the nonautonomous setting to make sure that minimizers are Lipschitz continuous.

A well established approach consists in imposing superlinearity together with some regularity conditions on the state or velocity variables in order to ensure the validity of both the Euler condition and Weierstrass inequality, see \cite{ClarkeMemoirs} for a minimal set of assumptions.

Alternatively, one can impose a local Lipschitz condition just on the time variable  of the Lagrangian, that we call  here  Condition (S) (see \S~\ref{sect:S}). Condition (S) was known in the smooth setting for providing the validity of the Du Bois-Reymond equation (see \cite{Cesari}). In the nonsmooth setting it became a key assumption for several recent results concerning important aspects such as existence and regularity  of minimizers:
\begin{itemize}
\item Existence: Clarke
introduced in his seminal paper \cite{Clarke1993} the essential idea of using an indirect weak growth condition, named henceforth of type (H), including both  Lagrangians of the form $\L(s,y,u)=\sqrt{1+|u|^2}$, and  superlinear ones. In \cite{Clarke1993} it is shown that Condition (S) with  Condition (H) allow to
 replace the superlinearity assumption in Tonelli's existence theorem  (leaving unchanged lower semicontinuity of the Lagrangian and convexity in the velocity variable), with the advantage that minimizers turn out to be Lipschitz.
\item  Regularity: Condition (S) alone yields the validity of a Du Bois-Reymond (DBR) type condition expressed in terms of convex subdifferentials, without any convexity assumption (see \cite{BM2, BM1}). The fact that (S) is satisfied whenever the Lagrangian is autonomous implies in particular  the validity of the (DBR) condition for any Borel autonomous Lagrangian.  Once Condition (S) is fulfilled, the weak growth condition (H) (alone if $\L$ is real valued)  yields the Lipschitz continuity of the minimizers, when they exist, see \cite{BM2}.
\end{itemize}
Conditions such as (H) and (S) can be rephrased in the context of optimal control, providing Lipschitz regularity of minimizers and boundedness of optimal controls (cf. \cite{CFM}, \cite{FrankTrans}, \cite{BM3}, \cite{MTrans}).

We study here the problem of finding a {\em uniform} Lipschitz constant for minimizers of a Bolza type control problem with variable endpoint of the form
\begin{equation}\min J_{t}(y,u):=\int_t^T\L(s,y(s), u(s))\,ds+g(y(T))\tag{P$_{t,x}$}\end{equation}
\text{\textbf{Subject to:} }
\begin{equation}\label{tag:admissibleintro}\tag{D}\begin{cases}
y\in \AC([t,T];\R^n)\\y'=b(y)u\text{ a.e. } s\in [t,T], \,y(t)=x\\u(s)\in \mathcal U\text{ a.e. } s\in [t,T],\, y(s)\in \mathcal S\,\,\forall s\in [t,T],
\end{cases}
\end{equation}
%
as the initial time $t$ and point $x$ vary on compact sets.  A motivation is the study of the regularity of the value function, when one can assume the existence of an optimal pair for any initial data. This existence hypothesis on minimizers is widespread in the literature and becomes a starting point to derive properties on the value function, see for instance Dal Maso -- Frankowska \cite{DMF} in the autonomous and superlinear case of the calculus of variations
 In the real valued case our main result, Theorem~\ref{th:motivH} below, states that if $\L$ satisfies Condition (S) and a growth condition of type (H), then the minimizers of ($P_{t,x}$) are equi-Lipschitz whenever the $t, x$ belong to a compact set. Furthermore, if one knows an a priori upper bound of the integral terms $\int_t^T\L(s,y_*(s), u_*(s))\,ds$ along
the minimizers, a common Lipschitz rank may be explicitly written.
We shall consider also the case of the extended valued Lagrangians: in this case some further assumptions, namely lower semicontinuity of $\L(s,y,u)$ with respect to $(s,u)$ and a topological property of the effective domain of $\L$, are needed in order to prove the regularity result on minimizers.

 The  growth condition introduced in \S~\ref{subsect:H} represents a violation of the (DBR) condition for high values of the velocity; it
  coincides with Clarke's original one when the compact set is reduced to a single initial datum $(t_0,x_0)$ and the Lagrangian is convex in the velocity variable.  In the case of an extended valued Lagrangian, it is new and includes the class of functions considered in \cite{BM3}, where minimizers regularity is obtained for a single optimal pair without necessarily deriving any kind of uniformity for initial data in a compact set:
the uniform regularity result established here covers the class of Lagrangians that satisfy the assumptions employed for \cite[Theorem 4.2]{BM3}.

  As a byproduct of our formulation, the growth condition (Condition (G), see \S~\ref{sect:G})  introduced by Cellina -- Treu -- Zagatti in \cite{CTZ} and studied in \cite{Cellina,  CF1, MTLip} becomes a particular case of the class of problems considered here.\\
An equi-Lipschitz minimizers regularity was recently established in \cite{MTrans} under the additional assumption that $0<r\mapsto \L(s,y,ru)$ is convex for all $u$ (called `radial convexity'); in our paper we consider problems which may be not necessarily radially convex.

 Moreover, differently from \cite{MTrans},  minimizers may just be local ones in the sense of the absolutely continuous norm. The fundamental tool in the proof of Theorem~\ref{th:motivH} is the Du Bois-Reymond condition established in \cite[Theorem 3.1]{BM3}.

As an application, we extend the local Lipschitz regularity of the value function formulated in \cite{DMF} in the framework of autonomous and  superlinear Lagrangians to the nonautonomous ones under the slower growth condition of type (H).

\section{Preliminaries}
\subsection{Basic setting and notation}
Let $ t< T$ and $x\in\R^n$. We consider the Bolza type \textbf{optimal control problem}
\begin{equation}\min J_{t}(y,u):=\int_t^T\L(s,y(s), u(s))\,ds+g(y(T))\tag{P$_{t,x}$}\end{equation}
\text{\textbf{Subject to:} }
\begin{equation}\label{tag:admissible}\tag{D}\begin{cases}
y\in \AC([t,T];\R^n)\\y'=b(y)u\text{ a.e. } s\in [t,T], \,y(t)=x\\u(s)\in \mathcal U\text{ a.e. } s\in [t,T],\, y(s)\in \mathcal S\,\,\forall s\in [t,T],
\end{cases}
\end{equation}
with the following basic assumptions.

\begin{basicass} The following conditions hold {\rm(}$n, m\ge 1${\rm)}.
\begin{itemize}
\item $t_0<T$ are given real numbers, and $t\in [t_0,T]$;
\item The \textbf{Lagrangian}
$\L:[t_0, T]\times\R^n\times\mathcal \R^m\to  \R \cup\{+\infty\}$, $(s,y,u)\mapsto\L(s,y,u)$  is $Lebesgue-Borel$ measurable (i.e., measurable with respect to the $\mathcal L([t_0,T])\times \mathcal B_{\R^n\times\R^m}$ measurable sets);
\item $b:\R^n\to L(\R^n,\R^m)$ (the space of linear functions from $\R^n$ to $\R^m$) is a Borel measurable function such that, for some $\theta\ge 0$,
    \begin{equation}\label{tag:assb}|b(y)|\le \theta(1+|y|).\end{equation}  We refer to $y'=b(y)u$ as the \textbf{controlled differential equation};
\item The \textbf{control} $u:[t,T]\mapsto\R^m$ is measurable;
\item The state  constraint set $\mathcal S$ is a nonempty subset of $\R^n$;
\item The \textbf{control set} The set $\mathcal U\subset\R^m$ is a cone, i.e. if $u\in \mathcal U$ then $\lambda u\in \mathcal U$ whenever $\lambda>0$.
\item The \textbf{cost function}
$g:\R^n\to \R\cup \{+\infty\}$ is not identically equal to $+\infty$.
    \item \textbf{(Linear growth from below)} There are $\alpha>0$ and $d\ge 0$ satisfying, for a.e. $s\in [t_0,T]$ and every $y\in\R^n, u\in \mathcal U$,
\begin{equation}\label{tag:lingrowth}\L(s,y,{u})\ge \alpha|{u}|-d.\end{equation}
\end{itemize}
\end{basicass}\noindent
An \textbf{admissible pair} for {\rm (P}$_{t,x}${\rm )} is a pair of functions $(y,u):[t,T]\to\R^n\times\R^m$ with $u$ measurable, $(y,u)$ satisfying \eqref{tag:admissible} and such that $J_t(y,u)<+\infty$.
We assume henceforth that, for each $t\in [t_0, T]$ and $x\in \mathcal S$,  there exists at least an admissible pair  for {\rm (P}$_{t,x}${\rm )}.\\
\noindent
\\
Notice, that in the particular case where  the function  $b\equiv 1$ in the controlled differential equation,  then {\rm (P}$_{t,x}${\rm )} becomes a problem of the \textbf{Calculus of Variations}.\\
If $z\in\R^k$ we shall denote by $B^k_{r}(z)$ (simply $B_r^k$ if $z=0$) the closed ball  of center $z$ and radius $r$ in $\R^k$. The norm in $L^1$ is denoted by $\|\cdot\|_1$, and the norm in $L^{\infty}$ by $\|\cdot\|_{\infty}$.\\
\subsection{Condition {\rm {\rm(S)}}}\label{sect:S}
We will consider the following local Lipschitz condition on the Lagrangian $\L$ with respect to the time variable.
\begin{conditionS}
There are $\kappa,  A\ge 0, \gamma\in L^1([t_0,T])$,
$\varepsilon_*>0$ satisfying,
for a.e.  $s\in [t_0,T]$
 \begin{equation}\label{tag:H3}
|\Lambda(s_2,y,u)-\Lambda(s_1,y,u)|\le \big(\kappa\L(s,y,u)+A|u|+\gamma(s)\big)\,|s_2-s_1|
\end{equation}
whenever $s_1,s_2\in [s-\varepsilon_*,s+\varepsilon_*]\cap [t_0, T]$, $y\in\R^n$, $u\in\R^m$, are such that $(s_1,y,u), (s_2,y,u)\in\Dom(\Lambda)$.
\end{conditionS}
\begin{remark}\label{rem:auton55}
Condition {\rm (S)} is satisfied if $\L(s,y,u)=\L(y,u)$ is {\em autonomous}. Indeed in that case \eqref{tag:H3} holds with $\kappa=A=0$,$\gamma\equiv 0$ and $\varepsilon_*=T$.
\end{remark}
%


\section{Growth conditions}
The definitions and results in this section are similar to those ones which have been introduced in some recent papers (see \cite{BM2, BM3} and \cite{MTrans}). There are however some differences: the present definition of Condition (H$_B^{\delta}(\chi)$) is more general than the corresponding growth condition used in \cite{BM2, BM3}, and we do not require, as in \cite{MTrans}, that the Lagrangian is radially convex in the control variable. Therefore, the detailed proofs are reported below for the convenience of the reader.
\subsection{Partial derivatives and subgradients}
In what follows we often deal with  subdifferentials in the sense of convex analysis.
\\
\begin{Notation}
If $(s,y,u)\in \Dom (\L)$, we shall denote by
\begin{itemize}
\item  $\partial_{\mu}\left(\L\Big(s,y,\dfrac{{u}}{\mu}\Big){\mu}\right)_{{\mu}=1}$ the \textbf{convex subdifferential}  of the map  \[0<{\mu}\mapsto \L\Big(s,y,\dfrac{{u}}{\mu}\Big){\mu}\] at ${\mu}=1$;
    \item
 $\partial_r\L\big(s,y,r{u}\big)_{r=1}$ the \textbf{convex subdifferential}  of the map  \[0<r\mapsto \L(s,y,r{u})\] at $r=1$;
 \item
 $\nabla_u\L(s,y,u)$ the \textbf{gradient} of $\L(s,y,\cdot)$ at  $u$. If $\L(s,y,\cdot)$ is differentiable then the (classical) derivative of $\L$ w.r.t. $u$ is written $D_u\L(s,y,u)=u\cdot\nabla_u\L(s,y,u)$.
 \end{itemize}
\end{Notation}

\begin{remark}\label{rem:variediff} Let $(s,y,u)\in \Dom(\L)$.
A simple change of variable $r=\dfrac1{\mu}$ shows
that
\[p\in \partial_{\mu}\Big(\L\Big(s,y,\dfrac{{u}}{\mu}\Big){\mu}\Big)_{{\mu}=1}\Leftrightarrow
\Lambda(s,y,{u})-p\in \partial_r\L\big(s,y,r{u}\big)_{r=1}.\]
\end{remark}
The  growth assumptions introduced below involve some uniform limits.
\subsection{The Growth Condition {\rm (G)}}\label{sect:G}
The growth Condition (G)  was thoroughly studied by Cellina and his school for autonomous Lagrangians of the calculus of variations that are smooth or convex in the velocity variable. The extension to the radial convex case, recalled here, was considered in \cite{MTLip} in the autonomous case and was introduced in \cite{BM1, BM3} for the nonautonomous case.
\begin{GrowthGloc} We say that $\L$ satisfies {\rm (G)} if, for all $K\ge 0$,
\begin{equation}\label{tag:GequivP}
\lim_{\substack{|{u}|\to +\infty\\ (s,y,u)\in\,\Dom(\L),\, u\in \mathcal U\\ P(s,y,{u})\in \partial_{\mu}(\L(s,z,\frac{{u}}{\mu}){\mu})_{{\mu}=1}
\not=\emptyset}}\!\!\!\!\!\!P(s,y,{u})=-\infty\,\,\,\text{unif.  } |y|\le K,\end{equation}
meaning that
for all $M\in\R$ there exists $R>0$ such that
$P(s,y,u)\le M$ for all $(s,y,u)\in\Dom(\L)$ with $\partial_{\mu}(\L(s,z,\frac{{u}}{\mu}){\mu})_{{\mu}=1}
\not=\emptyset$, $|y|\le K,\, u\in \mathcal U,\, |u|\ge R.$
\end{GrowthGloc}
\begin{remark}\label{rem:derivative}
\begin{enumerate}
\item If $u\mapsto\L(s,y,u)$ is differentiable, \eqref{tag:GequivP} becomes
\begin{equation}\label{tag:Gdiff33}\displaystyle\lim_{\substack{|{u}|\to +\infty\\ (s,y,u)\in\,\Dom(\L),\, u\in \mathcal U\\ \partial_{r}\L(s,z,ru)_{r=1}
\not=\emptyset}}\!\!\!\!\L(s,y,u)-u\cdot \nabla_{u}\L(s,y,u)=-\infty \text{ unif. }|y|\le K.\end{equation}
\end{enumerate}
\end{remark}

Superlinearity plays a key role in Tonelli's existence theorem. It has been widely used as a sufficient condition for Lipschitz regularity of minimizers.
\begin{AssumptionG}
There exists $\Theta:]-\infty, +\infty[\to \R$ such that, for a.e.  $s\in [t_0,T]$ and every $y\in \R^n$ $u\in\mathcal U$,
\[\Lambda(s,y,{u})\ge\Theta(|{u}|)\quad\forall {u}\in \R^n,\quad
\displaystyle\lim_{r\to +\infty}\dfrac{\Theta(r)}{r}=+\infty .\tag{G$_\Theta$}\label{tag:superlinearity}
\]
\end{AssumptionG}
Superlinearity, together with some local boundedness condition, implies the validity of the growth Condition {\rm (G)}. We refer to \cite[Proposition 2 and Remark 11]{BM2} for the proof of the following result.
\begin{proposition}[\textbf{Superlinearity  $\Rightarrow$ {\rm \textbf{(G)}}}]\label{prop:SuperimpliesG}Let $\L$ be superlinear,  and assume that
there is $r_0>0$ such that $(s,y,u)\in\Dom(\L)$ whenever $s\in [t_0,T], y\in\R^n$ and $u\in\R^m$ with $|u|= r_0$.
 Then
$\L$ satisfies
Assumption {\rm (G)}.
\end{proposition}
\subsection{Assumptions on $\Dom(\L)$ and distance-like functions}
 The  {\bf effective domain} of $\L$, given by
   \[\Dom (\L):=\{(s,y,u):\, \L(s,y,u)<+\infty\}.\]
We assume that for a.e. $s\in [t_0, T]$ and every $y\in\R^n$ the set \[\{u\in\R^n:\,(s,y,u)\in\Dom (\L)\}\]    is \textbf{strictly star-shaped in the variable $u$ w.r.t.  the origin}, i.e.,
\begin{equation}\L(s,y,{u})<+\infty, \, 0<r\le 1\Rightarrow \L(s,y,r{u})<+\infty.\label{tag:starshaped}\end{equation}
\begin{definition}[$u$-distance, $\infty$-distance, Euclidean distance]\label{def:wellinside}\phantom{AA}
\begin{itemize}
\item We shall denote by $\dist_e$ the usual \textbf{Euclidean distance} in $[t_0, T]\times \R^n\times\R^m$.
\item
The \textbf{infinity distance} $\dist_{\infty}$ is defined
for all $\omega_i=(s_i, z_i, v_i) \in [t_0, T]\times \R^n\times\R^m\,(i=1,2)$,
\[
\dist_{\infty}(\omega_1, \omega_2)=\begin{cases} +\infty&\text{ if }\omega_1\not=\omega_2\\
0&\text{ if }\omega_1=\omega_2.\end{cases}\]
\item
The \textbf{$u$-distance} is the function defined on the pairs of points $\omega_1=(s_1, z_1, v_1), \omega_2=(s_2, z_2, v_2)\in [t_0, T]\times \R^n\times\R^m$ such that $(s_1, z_1)=(s_2, z_2)$ by
\[
\dist_{\infty}(\omega_1, \omega_2)=|v_2-v_1|.\]
\end{itemize}
If $\chi\in\{e, u, \infty\}$ and $(s,z,v)\in\Dom(\L)$ we set $\dist_{\chi}((s,z,v),\Dom(\L)^c)$ to be equal to
\[\inf\{\dist_{\chi}((s,z,v), \omega):\, \omega\in ([t_0,T]\times\R^n\times\R^m)\setminus\Dom(\L)\}.\]
\end{definition}
\begin{remark}\label{rem:triangular}
Differently from the Euclidean and infinity distances, the $u$-distance is not a metric on $[t_0, T]\times\R^n\times\R^m$. We point out, however, that as well as $\dist_e$ and $\dist_{\infty}$, $\dist_u$ satisfies the triangular inequality among triples of  points that have the same first two coordinates.  Notice also that if $\chi\in\{e,u\}$ then
\[\begin{aligned}\dist_{\chi}((s,y,u), \Dom(\L)^c)&=\dist_{\chi}((s,y,u), \partial\Dom(\L))\\&:=\inf\{\dist_{\chi}((s,z,v), \omega):\, \omega\in \partial\Dom(\L)\}.\end{aligned}\]
The above  is no more true if $\chi=\infty$ and $\partial\Dom(\L)\cap \Dom(\L)\not=\emptyset$.
\end{remark}
\begin{definition}[Well-inside $\Dom(\L)$ for $\dist_{\chi}, \chi\in\{e,u, \infty\}$]
 We say that a subset $A$ of $\Dom(\L)$ is
\textbf{well-inside $\Dom(\L)$} w.r.t. $\dist_{\chi} (\chi\in\{e, u, \infty\}$) if it is contained in  $\{(s,y,u)\in\Dom(\L):\, \dist_{\chi}((s,y,u),\,\Dom(\L)^c)\ge \rho\}$, for a suitable $\rho>0$.
\begin{itemize}
\item If $\chi=e$ the above means that for all $(s,y,u)\in A$, the open ball
 of radius $\rho$ in $I\times \R^n\times \R^m$ and center in $(s,y,u)$ is contained in $\Dom(\L)$;
\item If   $\chi=u$ the above means that
\[(s,y,u)\in A,\, 0<r<\rho\Rightarrow (s,y,u+r u)\in \Dom(\L).\]
\item If $\chi=\infty$ the above means simply that $A\subset \Dom(\L)$.
\end{itemize}
\end{definition}
\begin{remark}
Notice that, if $\omega:=(s,y,u)\in\Dom(\L)$ and $F:=\Dom(\L)^c$, then
\begin{equation}\label{tag:ineq}\dist_{e}(\omega, F)\le \dist_{u}(\omega, F)\le \dist_{\infty}(\omega, F).\end{equation}
Thus, if $\mathcal M_{\chi}$ is the class of sets that are well inside $\Dom(\L)$ w.r.t. $\dist_{\chi}$ we have
\begin{equation}\label{tag:inclusions}
\mathcal M_{e}\subset \mathcal M_{u}\subset \mathcal M_{\infty}.
\end{equation}
\end{remark}
\begin{example} Let $\L$ be autonomous and $\Dom(\L)=\{(y,u)\in\R^2:\, |y|\le 1\}$. Then the set $\{(y,u)\in\R^2:\, |y|\le 1, |u|\le 1\}$ is well-inside $\Dom(\L)$ w.r.t. to $d_u$ but not w.r.t. $d_e$.
\end{example}
\subsection{Growth Condition {\rm(H$_B^{\delta}(\chi)$)}}\label{subsect:H}
Let $\delta\in[t_0,T[$.
The number $B$ represents an upper bound of the integral term in (P$_{t,x}$) for a prescribed family of admissible pairs, with initial time $t$ varying in $[t_0,\delta]$. The following quantities $c_{t}(B)$ and $\Phi(B)$ will play a role in the proof of the main results.
\begin{definition}[\textbf{$c_{t}(B)$ and $\Phi(B)$}]\label{def:cphi}
Let $t\in [t_0,T[$, $B\ge  0$ and
assume the linear growth from below \eqref{tag:lingrowth}, i.e., for a.e. $s\in [t_0,T]$, for all $y\in\R^n, u\in\mathcal U$,
 \[\Lambda(s,y,u)\ge \alpha|{u}|-d\quad (\alpha>0, d\ge 0).\]
 Set
\[c_{t}(B):=\dfrac{B+d(T-t)}{\alpha\,(T-t)}.\]
Moreover, if   Condition {\rm (S)} holds, we  define
\begin{equation}
\Phi(B):=\kappa B+\dfrac A{\alpha}(B+d\,(T-t_0))+\|\gamma\|_{1},
\end{equation}
where we set $\kappa, A, \gamma$ equal to 0 if $\L$ is {\em autonomous}.
\end{definition}
\begin{remark}
Notice that, in Definition~\ref{def:cphi}, $c_t(B)\le c_{\delta}(B)$ for all $t\in [t_0,\delta]$. In the autonomous case, since $\kappa, A$ and $\gamma$ may be chosen to be equal to 0, we consider $\Phi(B):=0$ (see Remark~\ref{rem:auton55}).
\end{remark}
The next result highlights the roles  of $\Phi(B)$ and $c_t{\rm (B)}$, we refer to \cite[Proposition 4.10]{MTrans} for a proof.
\begin{proposition}[\textbf{The role  of $\phi{\rm (B)}$ and $c_t{\rm (B)}$}]\label{lemma:newC}
 Assume the linear growth from below \eqref{tag:lingrowth} and the validity of Condition {\rm (S)}.
Let $t\in[t_0,T[$, $x\in\R^n$, and take an admissible pair $(y,u)$ for {\rm (P}$_{t,x}${\rm )} with $\displaystyle\int_t^T\L(s,y(s), u(s))\,ds\le B$ for some $B\ge 0$.
Then
\begin{enumerate}
\item
\begin{equation}\label{tag:100}\int_t^T|u(s)|\,ds\le \dfrac{B+d(T-t)}{\alpha}=(T-t)c_t(B).\end{equation}
\item For every $c>c_{t}(B)$  the set
    $\{s\in [t,T]:\, |u(s)|<c\}$
    is non negligible.
\item
$
\displaystyle\int_t^T\big\{\kappa\L(s,y(s),u(s))+A|u(s)|+\gamma(s)\big\}\,ds\le \Phi(B).
$
\end{enumerate}
\end{proposition}
Given $B\ge 0$ and $\delta\in [t_0, T[$, the growth Condition (H$_B^{\delta}(\chi)$) below requires the validity of Condition {\rm (S)}, unless $\L$ is autonomous. It will be applied when $B$ is an upper bound for the values of a given set of admissible pairs for problems {\rm (P}$_{t,x}${\rm )} as $t\in [t_0, \delta]$. \\
Condition {\rm (H$_{B}^{\delta}(\chi)$)} is a refinement of \cite[Condition (H)]{Clarke1993}, introduced by Clarke, who first thoroughly began the investigation on existence and regularity under such a kind of indirect weak growth condition.\\
Below, taking the inf/sup where $P(s,y,{u})\in \partial_{\mu}(\L(s,z,\frac{{u}}{\mu}){\mu})_{{\mu}=1}
\not=\emptyset$   means that we consider just those points $(s,y,u)$ such that $\partial_{\mu}(\L(s,z,\frac{{u}}{\mu}){\mu})_{{\mu}=1}
\not=\emptyset$.
\begin{ConditionH}
Assume that $\L$ satisfies Condition {\rm (S)} and let $\chi\in\{e,u, \infty\}$.
Let $B\ge 0$ and $\delta\in [t_0,T[$.
We say that $\L$ satisfies {\rm (H$_{B}^{\delta}(\chi)$)} if for all $K\ge 0$,
  there are $\overline\nu>0$  and ${c}>c_{\delta}(B)$
  satisfying, for all $\rho>0$,
\begin{equation}\label{tag:Hequiv}
\sup_{\substack{s\in [t_0,T],|y|\le K\\|u|\ge \overline\nu, u\in \mathcal U\\ \L(s, y, u)<+\infty\\P(s,y,{u})\in \partial_{\mu}(\L(s,z,\frac{{u}}{\mu}){\mu})_{{\mu}=1}
\not=\emptyset}} \!\!\!\!\! \!\!\!\!\!
\{P(s,y,u)\}+\Phi(B)< \!\!\!\!\! \!\!\!\!\!\inf_{\substack{s\in [t_0,T],|y|\le K\\|u|<c, u\in \mathcal U\\ \L(s, y, u)<+\infty\\ \dist_{\chi}((s,y,u),\Dom(\L)^c)\ge\rho\\P(s,y,{u})\in \partial_{\mu}(\L(s,z,\frac{{u}}{\mu}){\mu})_{{\mu}=1}
\not=\emptyset}}
 \!\!\!\!\! P(s,y,u).\end{equation}
 %
\end{ConditionH}
\begin{remark} Condition (H$_{B}^{\delta}(\chi)$) was originally introduced with $\delta=t_0$ and $\chi=\infty$ in \cite{Clarke1993}, and subsequently considered in \cite{BM2} and no interest in investigating the uniformity of the Lipschitz constant of the minimizers as the initial time and datum vary.  Considering $\chi=e$ or $\chi=u$ enlarge the class of extended valued Lagrangians that satisfy (H$_{B}^{\delta}(\chi)$). Notice, in view of  \eqref{tag:inclusions}, that from \eqref{tag:Hequiv} we have
\[(H_{B}^{\delta}(\infty))\Rightarrow (H_{B}^{\delta}(u))\Rightarrow(H_{B}^{\delta}(e)).\]
We refer to \cite[Example 4.18]{MTrans} for a Lagrangian that satisfies $(H_{B}^{\delta}(e))$ but not $(H_{B}^{\delta}(\infty))$.
\end{remark}




\begin{remark}\label{rem:HH}
\begin{enumerate}
 \item The validity of Condition (H$_B^{\delta}(\chi)$)  implies that the right side of inequality \eqref{tag:Hequiv} is not equal to $-\infty$.
        \item If $u\mapsto\L(s,y,u)$ is differentiable, \eqref{tag:Hequiv} may be rewritten 
  as
\begin{multline}\label{tag:Hdiff2}
\!\!\!\!\!\!\!\!\!\!\!\!\!\!\!\!\sup_{\substack{s\in [t_0,T],|y|\le K\\|u|\ge \overline\nu, u\in \mathcal U\\ \L(s, y, u)<+\infty\\ \partial_{\mu}(\L(s,z,\frac{{u}}{\mu}){\mu})_{{\mu}=1}
\not=\emptyset}}
\!\!\!\!\!\!\!\!\!\!\!\!\!\!\!\!\{\L(s,y,u)-u\cdot\nabla_u\L(s,y,u)\}+\Phi(B)<
\\
<\inf_{\substack{s\in [t_0,T],|y|\le K\\|u|<c, u\in \mathcal U\\ \L(s, y, u)<+\infty\\ \dist_{\chi}((s,y,u),\Dom(\L)^c)\ge\rho\\ \partial_{\mu}(\L(s,z,\frac{{u}}{\mu}){\mu})_{{\mu}=1}
\not=\emptyset}}
\!\!\!\!\!\!\!\!\!\!\!\!\!\!\!\!\!\!\!\{\L(s,y,u)-u\cdot\nabla_u\L(s,y,u)\}.\end{multline}
    \end{enumerate}
\end{remark}
\begin{remark}[\textbf{Interpretation of  {\rm (G)}  and of {\rm (H$_B^{\delta}(\chi)$)}}]\label{rem:equivG} Consider for simplicity a Lagrangian $\L(u)$ of the variable $u$. Let $\L({u})<+\infty$ and assume that
\[P(u)\in\partial_{\mu}\Big(\L\Big( \dfrac{{u}}{\mu}\Big ){\mu}\Big)_{{\mu}=1}\not=\emptyset.\] Then $P(u)=\L(u)-Q(u)$ for some
 $Q({u})\in  \partial_{r}\L(r{u})_{r=1}$.
  Notice that
\[\L(r{u})\ge\phi_u(r):=\L({u})+ Q({u})(r-1)\quad \forall r>0.\]
 The value
$\phi_u(0)=P({u}):=\L({u})- Q({u})$ represents the intersection of the ``tangent'' line $z=\phi_u(r)$ to $0<r\mapsto \L(r{u})$ at $r=1$ with the $z$ axis.\\
Condition (G)  thus means that the ordinate $P( {u})$ of the above intersection point tends to $-\infty$ as $|{u}|$ goes to $\infty$.
\\
Condition {\rm (H$_B^{\delta}(\chi)$)} means that there is  a gap of  at least $\Phi(B)$ between the  above points as $|u|\ge\overline\nu$ and when evaluated at $u$ such that $|u|<c$, more precisely that
\[\sup_{|u|\ge\overline\nu}P(u)+\Phi(B)<\inf_{|u|<c}P(u).\]

\begin{figure}[htbp]
\begin{center}
\includegraphics[width=0.60\textwidth]{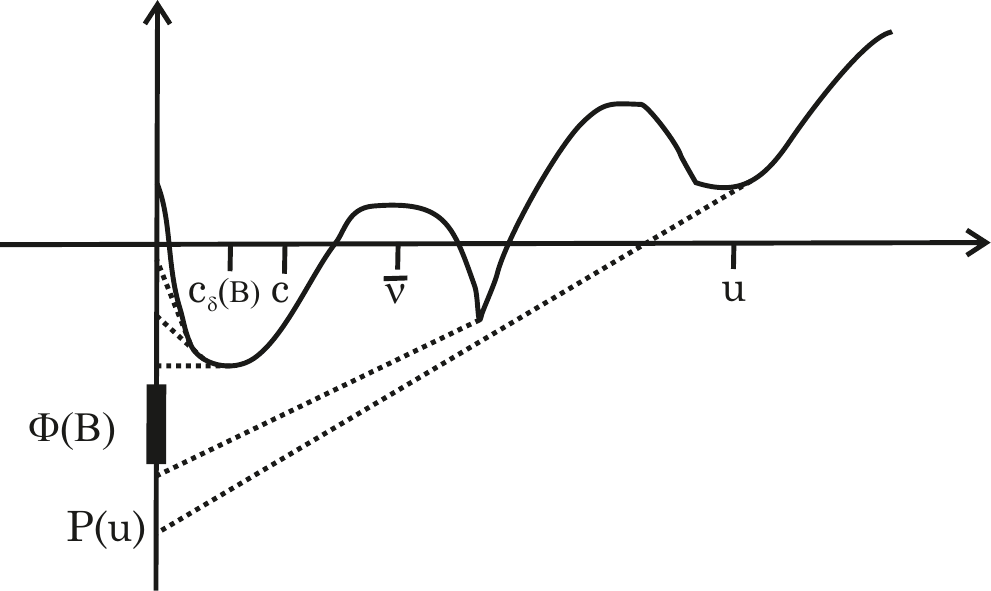}\
\end{center}
\caption{Condition {\rm (H$_{B}^{\delta}$)}}
\label{fig:G}
\end{figure}
\end{remark}
The validity of Condition (H$_B^{\delta}(\chi)$) implies that the infimum (resp. the sup) involved in \eqref{tag:Hequiv} is not equal to $-\infty$ (resp. $+\infty$). These facts, actually, occur quite often, independently of Condition (H$_B^{\delta}(\chi)$): their validity is actually a slow growth Condition, it was  introduced and named  (M$_B^{\delta}$) in \cite{MTrans}. Claim 2) of Proposition~\ref{prop:supinf} improves the sufficient condition formulated in \cite[Proposition 4.24]{MTrans}.
\begin{proposition}\label{prop:supinf} Let $K\ge 0$.
\begin{enumerate}
\item Assume that
 $\L$ is bounded on the bounded sets that are \textbf{\textbf{well-inside $\Dom(\L)$}} w.r.t. $\dist_{\chi} (\chi\in\{e, u, \infty\})$. For any  $c, \rho>0$,
\begin{equation}\label{tag:Pbound}-\infty<\inf_{\substack{s\in [t_0,T],|y|\le K\\|u|<c, u\in \mathcal U\\ \L(s, y, u)<+\infty\\ \dist_{\chi}((s,y,u),\Dom(\L)^c)\ge\rho\\P(s,y,{u})\in \partial_{\mu}(\L(s,z,\frac{{u}}{\mu}){\mu})_{{\mu}=1}
\not=\emptyset}}
P(s,y,u).\end{equation}
\item Assume  that there is $\nu>0$ such that
\begin{equation}\L \text{ is bounded on }([t_0,T]\times B_K^n\times B_{\nu}^m)\cap\Dom(\L).\tag{$\mathcal B$}\end{equation}
 Then
\begin{equation}\label{tag:Pbound}\sup_{\substack{s\in [t_0,T],|y|\le K\\|u|\ge \nu, u\in \mathcal U\\ \L(s, y, u)<+\infty\\P(s,y,{u})\in \partial_{\mu}(\L(s,z,\frac{{u}}{\mu}){\mu})_{{\mu}=1}
\not=\emptyset}}
P(s,y,u)<+\infty.\end{equation}
\end{enumerate}
\end{proposition}
\begin{proof}
1) Fix $c, \rho >0$. It is not restrictive to assume  that $\partial_{\mu}(\L(s,z,\frac{{u}}{\mu}){\mu})_{{\mu}=1}
\not=\emptyset$ for some $(s,y,u)\in\Dom(\L),  \dist_{\chi}((s,y,u),\Dom(\L)^c)\ge\rho$.
It follows from Remark~\ref{rem:variediff} that
\[\inf_{\substack{s\in [t_0,T],|y|\le K\\|u|<c, u\in \mathcal U\\ \L(s, y, u)<+\infty\\ \dist_{\chi}((s,y,u),\Dom(\L)^c)\ge\rho\\P(s,y,{u})\in \partial_{\mu}(\L(s,z,\frac{{u}}{\mu}){\mu})_{{\mu}=1}
\not=\emptyset}}P(s,y,u)=\!\!\!\!\!\!\!\!\!\!\inf_{\substack{s\in [t_0,T],|y|\le K\\|u|<c, u\in \mathcal U\\ \L(s, y, u)<+\infty\\ \dist_{\chi}((s,y,u),\Dom(\L)^c)\ge\rho\\Q(s,y,{u})\in \partial_{r}(\L(s,z,ru){\mu})_{{r}=1}
\not=\emptyset}}\{\L(s,y,u)-Q(s,y,u)\}.\]
The claim follows directly from Lemma~\ref{tag:boundabovewell}.\\
2)  Let $(s,y,u)\in\Dom(\L)$ with $|y|\le K$ and $|u|\ge {\nu}, u\in\mathcal U$. Assume that $P(s,y,u)\in \partial_{\mu}(\L(s,z,\frac{{u}}{\mu}){\mu})_{{\mu}=1}\not=\emptyset$. Then $P(s,y,u)=\L(s,y,u)-Q(s,y,y)$ for some $Q(s,y,u)\in  \partial_{r}(\L(s,z,ru))_{{r}=1}$ (Remark~\ref{rem:variediff}).
The assumption that $\Dom(\L)$ is star-shaped in the control variable implies that  $\Big(s, y, {\nu}\dfrac{u}{|u|}\Big)\in\Dom(\L)$ and thus
\begin{equation}
\L\Big(s, y, {\nu}\dfrac{u}{|u|}\Big)-\L(s, y, u)\ge Q(s,y,u)\Big(\dfrac{{\nu}}{|u|}-1\Big),
\end{equation}
from which we deduce that
\begin{equation}\label{tag:allconvM}
\L(s, y, u)-Q(s,y,u)\le \L\Big(s, y, {\nu}\dfrac{u}{|u|}\Big)-\dfrac{{\nu}}{|u|}Q(s,y,u).
\end{equation}
The assumptions imply that $\L\Big(s, y, {\nu}\dfrac{u}{|u|}\Big)\le C_1(K, \nu)$ for some  constant $C_1(K, \nu)$ depending only on $K, \nu$. \\
We now provide un upper bound for $-Q(s,y,u)$. The assumption that $\Dom(\L)$ is star-shaped in the control variable implies that  $\Big(s, y, \dfrac{\nu}2\dfrac{u}{|u|}\Big)\in\Dom(\L)$ and thus
\begin{equation}\label{tag:qigiug1}\L\Big(s, y, \dfrac{\nu}2\dfrac{u}{|u|}\Big)-\L(s, y, u)\ge Q(s,y,u)\left(\dfrac{\nu}{2|u|}-1\right),\end{equation}
so that the linear growth hypothesis (L)  gives
\begin{equation}\label{tag:qigiug2}\begin{aligned}-Q(s,y,u)&\le \dfrac{1}{\left(1-\dfrac{\nu}{2|u|}\right)}\left[\L\Big(s, y, \dfrac{\nu}2\dfrac{u}{|u|}\Big)-\L(s, y, u)\right]\\
&\le 2\left[\L\Big(s, y, \dfrac{\nu}2\dfrac{u}{|u|}\Big)+d\right]
\le C_2(K, \nu)
\end{aligned}\end{equation}
for some constant $C_2(K, \nu)$ depending only on $K$ and $\nu$.
It follows from \eqref{tag:qigiug1} -- \eqref{tag:qigiug2} that the right-hand side of \eqref{tag:allconvM} is bounded above by a constant depending only on $K$ and $\nu$.
\end{proof}
\begin{remark} Assumption ($\mathcal B$) in Proposition~\ref{prop:supinf} is a  known sufficient condition for the nonoccurrence of the Lavrentiev gap for positive autonomous Lagrangians of the calculus of variations (see \cite[Assumption (B)]{ASC}). Unsurprisingly, the more recent  Condition \eqref{tag:Pbound} plays a role in the avoidance of the Lavrentiev phenomenon (see \cite{MTrans}).
\end{remark}

\begin{lemma}[\textbf{Bound of $\partial_r\L(s, y, ru)_{r=1}$ on bounded sets}] \label{tag:boundabovewell}
Assume that $\L(s,y,u)$ is  \textbf{bounded on the bounded subsets that are well-inside  $\Dom(\L)$ w.r.t. $\dist_{\chi} (\chi\in\{e,u,\infty\}$)}. Let $$\Sigma:=\{(s,y,u)\in\Dom(\L):\, \partial_r\L(s, y, ru)_{r=1}\not=\emptyset\}, $$ and $Q$ be any function satisfying $Q(s,y,u)\in \partial_r\L(s, y, ru)_{r=1}$ for every $(s,y,u)\in \Sigma$.
Then $Q$ is bounded  on the bounded sets of $\Sigma$  that are well-inside  $\Dom(\L)$ w.r.t. $\dist_{\chi}$.
\end{lemma}

\begin{proof}
Let $(s,y,u)\in \Dom(\L)$ and  $Q(s,y,u)\in \partial_r\L(s, y, ru)_{r=1}\not=\emptyset$. Suppose that, for some $C>0, \rho>0$, $|y|+ |u|\le C$  and \[\dist_{\chi}((s,y,u),\Dom(\L)^c)\ge\rho.\] The triangular inequality (see Remark~\ref{rem:triangular}) implies that
\[\dist_{\chi}\left(\left(s,y,u+\dfrac{\rho}{2C}u\right), \Dom(\L)^c\right)\ge \dfrac{\rho}{2}.\] Assuming that
\[\partial_r\L(s, y, ru)_{r=1}\not=\emptyset\] we obtain
\[\L\left(s,y,u+\dfrac{\rho}{2C}u\right)-\L(s,y,u)\ge\dfrac{\rho}{2C} Q(s,y,u).\]
The  boundedness assumption of $\L$ implies  that $Q(s,y,u)$ is bounded above by a constant  depending only on $C$ and  $\rho$.
Similarly, from
\[\L\left(s,y,u-\dfrac{\rho}{2C}u\right)-\L(s,y,u)\ge -\dfrac{\rho}{2C} Q(s,y,u),\]
we deduce an upper bound for $Q$.
\end{proof}

The fact that the validity of Condition (G) implies that of Condition (H$_B^{\delta}(\chi)$) was proved in \cite{BM2} for real valued Lagrangians and in \cite[Proposition 4.21]{MTrans} under the additional  assumption that $0<r\mapsto\L(s,y,ru)$ is convex. Actually, the result holds true in greater generality.
\begin{proposition}[\textbf{{\rm (G)} implies (H$_B^{\delta}(\chi)$) for all $B, \delta$}]\label{prop:GimpliesH} Assume that $\L$ satisfies   Condition {\rm (S)} and that:
 $\L$ is bounded on the bounded subsets that are well-inside  $\Dom(\L)$ w.r.t. $\dist_{\chi} (\chi\in\{e,u,\infty\}$).
If $\L$ satisfies Condition (G) then
$\L$ satisfies Hypothesis (H$_B^{\delta}(\chi)$), whatever are the choices of $\delta\in [t_0, T[$, $c>0$  and $B\ge 0$.
\end{proposition}
\begin{proof}
 Take any  $K\ge 0$. Assume that
\[\lim_{\substack{|{u}|\to +\infty\\ (s,y,u)\in\,\Dom(\L),\, u\in \mathcal U\\Q(s,y,u)\in \partial_{r}(\L(s,z,r\,u))_{{r}=1}
\not=\emptyset}}\L(s,y, {u})-Q(s,y,{u})=-\infty\,\,\,\text{ unif. }|y|\le K.\]
Then we obtain 
\begin{equation}\label{tag:iuefgquogf}
\lim_{\nu\to +\infty}\sup_{\substack{s\in [t_0,T]\\|u|\ge \nu, u\in \mathcal U\\ \L(s, y, u)<+\infty\\Q(s,y,u)\in \partial_{r}(\L(s,z,r\,u))_{{r}=1}
\not=\emptyset}}
\{\L(s, y, u)-Q(s,y,u)\}=-\infty\text{ unif. }|y|\le K.
\end{equation}
It follows from 1) of 
Proposition~\ref{prop:supinf} that
 Condition (H$_B^{\delta}(\chi)$) is valid, for any choice of $B, c>0$, $\delta\in [t_0, T[$.
\end{proof}
\begin{remark}
In Proposition~\ref{prop:GimpliesH}, the assumption that $\L$ is bounded on bounded sets that are well-inside  $\Dom(\L)$ is not a merely technical hypothesis (see \cite[Example 4.25]{MTrans}).
\end{remark}

\section{Uniform regularity for optimal pairs}
We say that $(y_*, u_*)$ is a $W^{1,1}$-weak optimal pair for (P$_{t,x}$) if there is $\varepsilon>0$ such that  $J_t(y_*, u_*)\le J_t(y,u)$  for any admissible pair $(y,u)$ such that $\|y-y_*\|_{1}+\|y'-y_*'\|_1\le\varepsilon$.
In \cite[Theorem 4.2]{BM3} it is shown that, if $(y_*, u_*)$ is a $W^{1,1}$-weak optimal pair for (P$_{t,x}$) and Condition (H$^0_{J_{t}(y_*,u_*)})$ holds,  then $u_*$ is bounded and $y_*$ has a finite Lipschitz rank. We give here a sufficient condition under which the above bounds are uniform as the initial time $t$ varies in  an interval $[t_0 ,\delta]$ ($\delta\in [t_0, T[$) and the initial point $x$ varies in a compact set.
\begin{theorem}[\textbf{$L^{\infty}$ -- uniform boundedness for optimal controls and equi-Lipschitz rank of minimizers}]\label{th:motivH}
Assume that $\L$  takes values in $\R$ and  satisfies Assumption {\rm(S)}.
Fix $\delta\in [t_0, T[$, $\delta_*\ge 0$ and $x_*\in \R^n$. Let $(y_*, u_*)$ be a $W^{1,1}$-weak optimal pair for {\rm (P}$_{t,x}${\rm )} where $t\in [t_0,\delta]$, $x\in B^n_{\delta_*}(x_*)$, and $\displaystyle\int_t^T\L(s,y_*(s), u_*(s))\,ds\le B$ for a suitable $B\ge 0$. Assume
 that $\L$ satisfies the growth condition {\rm (H}$_{B}^{\delta}(\chi)${\rm )}.
 Then  $u_*$ is  bounded and   $y_*$ is Lipschitz with bounds and ranks depending only on $\delta, B, \delta_*, x_*$.\\
The same conclusion is valid when $\L$ takes values in $\R\cup \{+\infty\}$, provided that we impose also the following assumptions:
\begin{enumerate}
\item[a)] $(s,u)\mapsto\L(s,y,u)$ is lower semicontinuous for every $y$ with $(s,y,u)\in\Dom(\L)$;
\item[b)] For every $(s,y, u)\in \Dom(\L)$, the set $\{\lambda>0:\,\L(s,y,\lambda u)<+\infty\}$ is open;
\item[c)] For a.e. $s\in [t, T]$,  $\{(s, y_*(s), u_*(s))\}$ is well-inside $\Dom(\L)$ w.r.t. $\dist_{\chi}$, i.e., there is $\rho_s>0$ such that
\[\dist_{\chi}((s,y_*(s),u_*(s)), \Dom(\L)^c)\ge\rho_s.\]
    \end{enumerate}
\end{theorem}
\begin{remark}
When $\L$ is an extended valued function, in Theorem~\ref{th:motivH} we impose the additional assumptions a), b) and c). Condition c) is employed in the proof of Theorem~\ref{th:motivH}  (for the extended valued case) to take advantage of the information provided by `$\inf$'-term in \eqref{tag:Hequiv} of the growth Condition (H$_{B}^{\delta}(\chi)$), while assumptions a) and b) are used just to ensure the validity of the Du Bois-Reymond condition \cite[Theorem 2]{BM2}. Therefore, a) and b) can be removed and the regularity properties of Theorem~\ref{th:motivH} remain valid provided that the Du Bois-Reymond condition \cite[Theorem 2]{BM2} is in force. This is the case, for instance, when $\L$ is the indicator function of a (bounded) control set $U$ (cf. \cite[Remark 4]{BM2}).
\end{remark}
\begin{remark} Let $\chi$ be as in Hypothesis {\rm (H}$_{B}^{\delta}(\chi)${\rm )}. Then, in Theorem~\ref{th:motivH}: \begin{itemize}
\item
If $\chi=u$  then Assumption c) follows from Assumption b).
\item If $\chi=\infty$ Assumption c) is always satisfied.
\item
If $\chi=e$ Assumptions b) and c) are  fulfilled if $\Dom(\L)$ is open in $[t_0,T]\times\R^n\times \R^m$.
\item The validity of Assumption c) is ensured once, for a.e. $s\in [t, T]$,
\[\lim_{\dist_{\chi}\left((s,z,v),\Dom(\L)^c\right)\to 0}\L(s,z,v)=+\infty,\]
uniformly w.r.t. $z$ in compact sets, i.e., if for all compact $K\subset\R^n$ and $M\ge 0$ there is $\rho>0$ such that
\[\forall (s,z,v),\, z\in K, \, \dist_{\chi}\left((s,z,v),\Dom(\L)^c\right)\le \rho\Rightarrow \L(s,z,v)\ge M.\]
\end{itemize}
\end{remark}
\begin{proof}[Proof of Theorem~\ref{th:motivH}]
Let $\alpha, d$ be as in \eqref{tag:lingrowth} and $(y_*, u_*)$ be a $W^{1,1}$-weak optimal pair for (P$_{t,x}$).
From Point 1 of Proposition~\ref{lemma:newC} we have
\begin{equation}\label{tag:R}\displaystyle\int_{t}^T|u_*|\,ds\le\dfrac{B+d(T-t)}{\alpha}\le  R=R(B):=\dfrac{B+d(T-t_0)}{\alpha}.
\end{equation}
\noindent
\emph{Claim: There is $K:=K(\delta, B, \delta_*, x_*)$ such that $|y_*(s)|\le K$ for every $s\in [t,T]$.}
Indeed, for a.e.  $s\in [t,T]$,
\begin{equation}\label{tag:boundyprime}|y_*'(s)|\le \theta(1+|y_*(s)|)|u_*(s)|.\end{equation}
Gronwall's Lemma (see \cite[Theorem 6.41]{ClarkeBook}) and \eqref{tag:R} imply that, for all $s\in [t,T]$,
\[\begin{aligned}|y_*(s)-x|&\le \int_t^s\exp\left(\theta\int_{\tau}^s|u_*(r)|\,dr\right)
\theta|u_*(\tau)|(|x|+
1)\,d\tau
\\
&\le \theta Re^{R\theta}(|x|+1)\le \theta  Re^{R\theta}(|x_*|+\delta_*+1).
\end{aligned}\]
The claim follows from  the fact that $R$ depends on $B$, with
$$K=|x_*|+\delta_*+\theta Re^{R\theta}(|x_*|+\delta_*+1).$$
 Assumptions a), b) imply that the Lagrangian $\L$ satisfies  \cite[Hypothesis (S$^{\infty}_{(y_*, u_*)}$)]{BM3}.
The optimal pair $(y_*,u_*)$ satisfies the Du Bois-Reymond -- Erdmann condition formulated  in \cite[Theorem 3.1]{BM3}. In particular
\[\partial_{\mu}\left(\L\Big(s,y_*(s),\dfrac{{u_*(s)}}{\mu}\Big){\mu}\right)_{{\mu}=1}
\not=\emptyset\qquad  \text{a.e. }s\in [t,T]\]
and there is an absolutely continuous function $p\in W^{1,1}([t,T])$ such that
\begin{equation}\label{tag:ouhfoqehf}p(s)\in \partial_{\mu}\left(\L\Big(s,y_*(s),\dfrac{{u_*(s)}}{\mu}\Big){\mu}\right)_{{\mu}=1}
\qquad  \text{a.e. }s\in [t,T],\end{equation}
\begin{equation}\label{tag:iugdiuqgei}
|p'(s)|\le \kappa\L(s,y_*(s),u_*(s))+A|u_*(s)|+\gamma(s)
\quad  \text{a.e. }s\in [t,T].
\end{equation}
We consider  $P(s,z,v)\in\partial_{\mu}\Big(\L\Big(s,z,\dfrac{{v}}{\mu}\Big){\mu}\Big)_{{\mu}=1}$ such that \[p(s)=P(s,y_*(s), u_*(s))\quad \text{ a.e. }s\in [t,T].\]
Let $\overline\nu$ be such that \eqref{tag:Hequiv} holds, with $\rho, K$ as above.
It follows from Claim 1 of Proposition~\ref{lemma:newC} that there is a non negligible set of $\tau\in [t,T]$ satisfying $|u_*(\tau)|<c$ and $p(\tau)=P(\tau, y_*(\tau), u_*(\tau))$.
We fix such a $\tau$ and  set $\rho:=\dist_{\chi}((\tau,y_*(\tau),u_*(\tau)), \Dom(\L)^c)$; notice that Assumption c) implies that $\rho>0$. We have
\begin{equation}\label{tag:gqgdiqd}P(s,y_*(s), u_*(s))=p(\tau)+\int_{\tau}^sp'(s)\,ds\quad\text{ a.e. } s\in [t,T].\end{equation}
It follows from  \eqref{tag:iugdiuqgei} and \eqref{tag:gqgdiqd} that for a.e. $s\in [t, T]$ we have
\begin{equation}\label{tag:ieqgfiqf}\begin{aligned}
p(\tau)&=P(s,y_*(s), u_*(s))-\int_{\tau}^sp'(s)\,ds\\
&\le P(s,y_*(s), u_*(s))+\int_{\tau}^s \left[\kappa\L(s,y_*(s),u_*(s))
+A|u_*(s)|+\gamma(s)\right]\,ds.\end{aligned}
\end{equation}
Assume that there is a non negligible subset $F$ of $[t,T]$ such that $|u_*|>\overline\nu$ on $F$. By taking $s\in F$ 
we deduce that
\begin{equation}\label{tag:CR1}\begin{aligned}
p(\tau)&\le
\!\!\!\!\!\sup_{\substack{s\in [t_0,T],|z|\le K\\|v|\ge \overline\nu, v\in \mathcal U\\ \L(s, z, v)<+\infty\\ \partial_{\mu}(\L(s,z,\frac{{v}}{\mu}){\mu})_{{\mu}=1}
\not=\emptyset}}
\!\!\!\!\!\!\{P(s,z,v)\}+\\&\phantom{AAAAA}+\left|\int_{\tau}^s \kappa\L(s,y_*(s),u_*(s))+A|u_*(s)|+\gamma(s)\,ds\right|\\
&\le
\sup_{\substack{s\in [t_0,T],|z|\le K\\|v|\ge \overline\nu, v\in \mathcal U\\ \L(s, z, v)<+\infty\\ \partial_{\mu}(\L(s,z,\frac{{v}}{\mu}){\mu})_{{\mu}=1}
\not=\emptyset}}
\{P(s,z,v)\}+\Phi(B),
\end{aligned}\end{equation}
 where the last inequality is justified by   Claim  2 of Proposition~\ref{lemma:newC}.
Now,
 \begin{equation}\label{tag:CR2}p(\tau)=P(\tau, y_*(\tau), u_*(\tau))
 \ge \!\!\!\!\!\!\!\!\!\!\!\!\!\!\inf_{\substack{s\in [t_0,T],|z|\le K\\|v|<c, v\in \mathcal U,  \L(s, z, v)<+\infty\\ \dist_{\chi}((s,z,v),\Dom(\L)^c)\ge\rho\\
\\ \partial_{\mu}(\L(s,z,\frac{{v}}{\mu}){\mu})_{{\mu}=1}
\not=\emptyset}}
\!\!\!\!\!\!P(s,z,v).\end{equation}
Therefore \eqref{tag:CR1} and \eqref{tag:CR2} imply that
\begin{equation}\label{tag:HnonconvN}
\!\!\!\!\!\!\!\!\!\!\!\!\!\!\!\!\!\!\!\!\sup_{\substack{s\in [t_0,T],|z|\le K\\|v|\ge \overline\nu, v\in \mathcal U,\, \L(s, z, v)<+\infty\\ \partial_{\mu}(\L(s,z,\frac{{v}}{\mu}){\mu})_{{\mu}=1}
\not=\emptyset}}
\!\!\!\!\!\!\{P(s,z,v)\}+\Phi(B)> \inf_{\substack{s\in [t_0,T],|z|\le K\\|v|<c, v\in \mathcal U,\,  \L(s, z, v)<+\infty\\\dist_{\chi}((s,z,v), \Dom(\L)^c)\ge\rho\\
\\ \partial_{\mu}(\L(s,z,\frac{{v}}{\mu}){\mu})_{{\mu}=1}
\not=\emptyset}}
\!\!\!\!\!\!\!\!\!\!\!\!\!\!\!P(s,z,v),
\end{equation}
contradicting \eqref{tag:Hequiv}. It follows that $|u_*|\le\overline\nu$ a.e. on $[t,T]$. The Lipschitzianity of $y_*$ and the uniformity of its rank follows from \eqref{tag:assb}.
\end{proof}
\begin{remark} The proof of Theorem~\ref{th:motivH} shows that if $\L$ is real valued then a uniform bound for the optimal control $u_*$ satisfying the conditions of the claim is given by any $\overline \nu>0$ satisfying one of the assumptions of Condition (H$_B^{\delta}(\chi)$), with
$K=|x_*|+\delta_*+\theta Re^{R\theta}(|x_*|+\delta_*+1)$ and $R=\dfrac{B+d(T-t_0)}{\alpha}$.
\end{remark}
One of the assumptions of Theorem~\ref{th:motivH} is the existence of an upper bound $B$ for the cost of the optimal pairs.
Such a bound exists and can be explicitly computed for some classes of problems,  e.g., for finite valued Lagrangians of the calculus of variations, or if the cost function $g$ is real valued.
Corollary~\ref{coro:unibound} below extends \cite[Proposition 3.3]{DMF} in various directions: Nonautonomous Lagrangians,  weaker growths than superlinearity, optimal control problems more general than problems of the calculus of variations, no convexity in the velocity variable.
\begin{corollary}[\textbf{The Calculus of variations or  real valued final cost $g$}]\label{coro:unibound}
Assume that  $\Lambda$ is \textbf{finite valued}, satisfies Assumption {\rm(S)} and is bounded on bounded sets.
Suppose that at least one of the following two assumptions holds:
\begin{enumerate}
\item \textbf{$b=1$ in the controlled differential equation}, $\mathcal S$ is convex and $\mathcal U=\R^n$, 
    \item the cost function \textbf{$g$ is real valued, locally bounded, bounded from below} and ($0\in \mathcal U$ for a.e. $s\in [t_0,T]$) or ($\mathcal S=\R^n$).
        \end{enumerate}
Let $\delta\in [t_0, T[$, $\delta_*\ge 0$,  $x_*\in \R^n$ and $\mathcal A$ be a set of optimal pairs for {\rm (P}$_{t,x}${\rm )} as $t\le\delta$ and $x\in B^n_{\delta_*}(x_*)$.
Assume that  $\L$ satisfies the growth condition {\rm (H}$_{B}^{\delta}(\chi)${\rm )} for every $B\ge 0$.
Then if  $(y_*,u_*)$ is an optimal pair in $\mathcal A$, $u_*$ is uniformly bounded and $y_*$ is uniformly Lipschitz.
\end{corollary}
\begin{proof}
From \cite[Lemma 5.3]{MTrans} we know that there is $B\ge0$ depending only on $\delta, \delta_*, x_*$ such that $\displaystyle\int_t^T\L(s,y(s), u(s))\,ds\le B$.
Theorem~\ref{th:motivH} yields the conclusion.
\end{proof}

\section{Lipschitz continuity of the value function}
We consider here  problem (P$_{t,x}$) in the framework of the \textbf{calculus of variations}, i.e., with $b\equiv 1$ in (D). The \textbf{value function} $V(t,x)$ associated  with problem (P$_{t,x}$) is the function defined by
\[\forall  t\in [t_0,T], \forall x\in \R^n\qquad V(t,x)=\inf{\rm (P}_{t,x}{\rm)}.\]
 In this section we shall assume that $\L$ is \textbf{finite valued} and \textbf{bounded on bounded sets}: since $g$ is not identically $+\infty$ it follows that $V(t,x)<+\infty$ for every $(t,x)$.
The next result extends to the nonautonomous case \cite[Corollary 3.4]{DMF}, formulated there for autonomous and superlinear Lagrangians.
\begin{corollary}[Local Lipschitz continuity of the value function]\label{coro:nonconvexcase}
Assume that  $\Lambda$ is \textbf{finite valued}, satisfies Assumption {\rm(S)} and is \textbf{bounded on bounded sets}. Suppose that $\L$ satisfies 1) or 2) of Corollary~\ref{coro:unibound} and the growth condition {\rm (H}$_{B}^{\delta}(\chi)${\rm )} for every $B\in \R$,  $\delta\in [t_0, T[$.
Assume, moreover,  that (P$_{t,x}$)\textbf{ admits a solution }for every $t\in [t_0,T]$ and $x\in\mathbb R^n$.
Then the value function $V(t,x)$ is locally Lipschitz on $[t_0, T[\times \R^n$.
 \end{corollary}
 \begin{remark}Sufficient conditions for the existence of a minimizer under the slow growth condition of type (H), required in Corollary~\ref{coro:nonconvexcase}, are provided in \cite{Clarke1993, MTrans}.
 \end{remark}
%
\begin{proof}[Proof of Corollary~\ref{coro:nonconvexcase}] Let  $x_*\in\R^n$ and $t_*\in[t_0, \delta[$ be given, for some $\delta\in ]t_0,T[$. Fix  $0<\varepsilon<T-\delta$ and take any $t_1, t_2 \in [t_*-\varepsilon/5, t_*+\varepsilon/5]\cap [t_0, \delta[$ and  any $x_1, x_2\in B^n_{\varepsilon/5}(x_*)$ with either $t_2\not=t_1$ or $x_2\not=x_1$. Set $\Delta:=|t_2-t_1|+|x_2-x_1|$.
Notice that  \[t_1<t_1+\Delta\le t_*+\varepsilon,\quad t_2\le t_1+\Delta.\]
Since $\inf$(P$_{t_2, x_2}$) is attained, let $y_2\in W^{1,1}([t_2,T];\R^n)$ be   such that \[y_2(t_2)=x_2,\quad J_{t_2}(y_2, y_2')= V(t_2, x_2).\]
From Corollary~\ref{coro:unibound} (in which we take $\delta_*=\varepsilon/5$), we know that every minimizer $y$ for ${\rm (P}_{t,x}{\rm)}$, for all $t \in [t_*-\varepsilon/5, t_*+\varepsilon/5]\cap [t_0, \delta[$ and $x\in B^n_{\varepsilon/5}(x_*)$, is
such that $||y||_\infty, \; ||y'||_\infty \le K$,
where the constant $K$ depends only on $\delta$, $\varepsilon$ and $x_*$.\\
Let
\[u:=\dfrac{y_2(t_1+\Delta)-x_1}{\Delta}.\]
The choice of $\varepsilon$ yields
\[\begin{aligned}|u|&\le \dfrac{|y_2(t_1+\Delta)-y_2(t_2)|}{\Delta}+\dfrac{|y_2(t_2)-x_1|}{\Delta}\\&
\le  \dfrac{|y_2(t_1+\Delta)-y_2(t_2)|}{\Delta}+\dfrac{|x_2-x_1|}{\Delta}\\&\le K\dfrac{|t_1+\Delta-t_2|}{\Delta}+\dfrac{|x_2-x_1|}{\Delta}\le K\dfrac{|\Delta|+|t_2-t_1|}{\Delta}+1\le2K+1.\end{aligned}\]
We consider now the competitor $z$ for (P$_{t_1, x_1}$) given by
\[z(s):=\begin{cases}x_1+(s-t_1)u\,&t_1\le s\le t_1+\Delta,\\
y_2(s)\, &t_1+\Delta\le s \le  T.
\end{cases}\]
Since $z(T)=y_2(T)$ we get
\begin{equation}\label{tag:finalestimate3}\begin{aligned}
V(t_1,x_1)&\le \int_{t_1}^{t_1+\Delta}\L(s,z,z')\,ds+\int_{t_1+\Delta}^T\L(s,y_2,y_2')\,ds+g(y_2(T))\\
&\le \int_{t_1}^{t_1+\Delta}\L(s,z,z')\,ds+V(t_2,x_2)-\int_{t_2}^{t_1+\Delta}\L(s,y_2,y_2')\,ds.
\end{aligned}\end{equation}
Since $0\le \Delta\le 4\varepsilon/5$ for all $s\in [t_1,t_1+\Delta]$ we obtain
\[|z(s)|\le |x_1|+\Delta|u|\le |x_*|+\varepsilon/5+4(2K+1)\varepsilon/5, \quad |z'(s)|\le |u|\le 2K+1,\]
so that, given that  $\L$ is bounded on bounded sets,
\[\left|\int_{t_1}^{t_1+\Delta}\L(s,z,z')\,ds\right|\le C(\varepsilon, K, x_*)\Delta= 2C(\varepsilon, K, x_*)(|t_2-t_1|+|x_2-x_1|),\]
 for some positive constant $C(\varepsilon, K, x_*)$ which depends only on $\varepsilon$, $K$, and $x_*$.
Moreover, as observed above, from Corollary~\ref{coro:unibound} we obtain that $||y_2||_\infty, \; ||y_2'||_\infty\le K$,
and thus, using the fact that  $|t_2-t_1|+\Delta\le 2\Delta$  (we can take the same constant $C(\varepsilon, K, x_*)$ previously employed)
\[\left|\int_{t_2}^{t_1+\Delta}\L(s,y_2,y_2')\,ds \right|\le  C(\varepsilon, K, x_*)
|t_1+\Delta-t_2|\le 2C(\varepsilon, K, x_*)(|t_2-t_1|+|x_2-x_1|).\]
It follows from \eqref{tag:finalestimate3} that
\[ V(t_1,x_1)-V(t_2,x_2)\le 4C(\varepsilon, K, x_*)(|t_2-t_1|+|x_2-x_1|).\]
  Exchanging the roles of $(t_1,x_1)$ and $(t_2,x_2)$ we arrive at
\[ |V(t_1,x_1)-V(t_2,x_2)|\le 4 C(\varepsilon, K, x_*)(|t_2-t_1|+|x_2-x_1|),\]
which proves the locally Lipschitz regularity of $V$ near $(t_*, x_*)$.
\end{proof}

\section*{Acknowledgments}
This research is partially supported by the  Padua University grant SID 2018 ``Controllability, stabilizability and infimum gaps for control systems'', prot. BIRD 187147. This research has been accomplished within the UMI Group TAA “Approximation Theory and Applications”.
\bibliographystyle{plain}
\bibliography{Unif_Bound}
\end{document}